\documentclass[10pt,oneside,a4paper]{amsart}
\usepackage{amsmath,amsthm, xy,latexsym,verbatim}
\usepackage{pxfonts}
\usepackage[psamsfonts]{amssymb}
\usepackage[latin1]{inputenc}
\usepackage{fancyhdr}
\usepackage{mathrsfs}
\usepackage{slashed}
\usepackage[bookmarks,colorlinks]{hyperref}
\usepackage{graphicx,emp}
\fancypagestyle{plain}{
	\fancyhead{}
	
}

\def\theequation{\thesection.\arabic{equation}}
\theoremstyle{definition}
\newtheorem{Def}{Definition}[section]
\theoremstyle{plain}
\newtheorem{Thm}{Theorem}[section]

\newtheorem{Prop}[Thm]{Proposition}

\newtheorem{rem}{Remark}

\newcommand{\R}{\mathbb{R}}

\newcommand{\C}{\mathbb{C}}

\newcommand{\Z}{\mathbb{Z}}
\newcommand{\N}{\mathbb{N}}

\newcommand{\id}{\textnormal{Id}}

\newcommand{\cl}{\textnormal{cl}}

\newcommand{\op}{\textnormal{Op}}
\newcommand{\wres}{\textnormal{wres}}
\newcommand{\res}{\textnormal{res}}

\renewcommand{\Re}{\mathfrak{Re}}

\def\SX{ {\mathcal{S}} }

\def\Dir{\slashed{D}}

\newcommand{\biindice}[3]%
{
\begin{array}[t]{c}
#1\\
{\scriptstyle #2}\\
{\scriptstyle #3}
\end{array}
}

\author{Ubertino Battisti}

\address{Dipartimento di Matematica, Università di Torino, Italy}

\email{ubertino.battisti@unito.it}

\author{Sandro Coriasco}

\address{Dipartimento di Matematica, Università di Torino, Italy}

\email{sandro.coriasco@unito.it}

\title[A Note on the Einstein-Hilbert action and Dirac operators on $\R^n$]{A Note on the Einstein-Hilbert action\\
and Dirac operators on $\R^n$}

\keywords{Wodzicki Residue, Einstein-Hilbert action, Dirac operator}

\subjclass[2000]{Primary: 58J40; Secondary: 58J42, 47A10, 47G30, 47L15}

\begin{document}

\begin{abstract}
We prove an extension to $\R^n$, endowed with a suitable metric, of the relation between the Einstein-Hilbert action and the 
Dirac operator which holds on closed spin manifolds.
By means of complex powers, we first define the regularised Wodzicki Residue for a class of operators globally defined on $\R^n$. The result is then obtained by using the properties of heat kernels and generalised Laplacians.
\end{abstract}

\maketitle

\section*{Introduction}
In 1984 M. Wodzicki \cite{WO84} introduced a trace on the algebra of classical pseudodifferential operators on a closed manifold $M$. A similar result was independently obtained by V. Guillemin \cite{GU85}, in order to give a \emph{soft} proof of Weyl formula. 
Wodzicki residue became then a standard tool in Non-commutative Geometry. In 1988, A. Connes \cite{CO88} proved that, for operator of order $-\dim(M)$, Dixmier Trace and Wodzicki residue are equivalent. Moreover, he conjectured that Wodzicki Residue could connect Dirac operators and Einstein-Hilbert actions on $M$. In 1995, D. Kastler \cite{KA95}, W. Kalau and M. Walze \cite{KW95} proved this conjecture. Namely, let $\Dir$ be the classical Atiyah-Singer operator defined on a closed spin manifold $M=(M,g)$ of even dimension $n\geq 4$. Then
\begin{equation}
\label{KKWTHM}
\wres(\Dir^{-n+2})= -\frac{(n-2) \, 2^{[\frac{n}{2} ]}}{\Gamma(\frac{n}{2}) (4\pi)^\frac{n}{2}}  \int_M \frac{1}{12}s(x)\,dx,
\end{equation}
where $s(x)$ is the scalar curvature and $dx$ the measure on $M$ induced by the Riemannian metric $g$
(see, e.g., \cite{AB02} for an overview about Wodzicki Residue and Non-commutative Geometry).
T. Ackermann \cite{AC96} gave a proof of \eqref{KKWTHM}, using the relationship between heat trace and $\zeta$-function and the properties of the second term in the asymptotic expansion of the heat trace of a generalised Laplacian.
Y. Wang \cite{WA06b}, \cite{WA06}, \cite{WA09} suggested an extension of the result to a class of manifolds with boundary.

In the last years, the definition of Wodzicki residue has been extended to other settings: manifolds with boundary \cite{FE96}, manifolds with conical singularities \cite{SC97}, \cite{LO02}, $SG$-calculus on $\R^n$ \cite{NI03}, anisotropic operators on $\R^n$ \cite{BN03}, and manifolds with cylindrical ends \cite{BC09}. In \cite{MSS06}, in order to study critical metrics on $\R^n$, it has been introduced the regularised trace of operators, which leads to the definition of regularised $\zeta$-function.
Wodzicki residue has been used also by R. Ponge \cite{PO08} to introduce \emph{lower dimensional volumes}. 

In this paper, we make use of $SG$-operators, a class of  pseudodifferential operators on $\R^n$ whose symbols $a(x,\xi)$, for fixed $\mu, m \in \R$ and all $x,\xi\in\R^n$, satisfy the estimates
\begin{equation}
\label{controllo}
|D_x^{\alpha} D_\xi^{\beta} a(x,\xi)|\leq C_{\alpha, \beta} (1+|x|)^{m-|\alpha|} (1+|\xi|)^{\mu-|\beta|}
\end{equation}
for suitable constants $C_{\alpha, \beta}\ge0$: the set of such symbols is denoted by $SG^{\mu, m}(\R^n)$.  If one deals with operators acting on sections of a vector bundle, the symbols are matrices whose entries must satisfy the inequalities \eqref{controllo}.
The corresponding operators can be defined via the usual left-quantisation
\[
A u(x)=\op(a)=\frac{1}{(2\pi)^{\frac{n}{2}}}\int e^{i x \xi} \, a(x,\xi) \, \hat{u} (\xi)\,d\xi,
\]
$u\in\SX(\R^n)$, and the set $\{\op(a)\mid a \in SG^{\mu, m}(\R^n)\}$ is denoted by $L^{\mu, m}(\R^n)$. Such operators form a graded algebra, that is $L^{\mu, m}(\R^n)\circ L^{\mu', m'}(\R^n)\subseteq L^{\mu+\mu', m+m'}(\R^n)$, map continuously $\SX(\R^n)$ to itself and can be extended as continuous operators from $\SX^\prime(\R^n)$ to itself. $L^{-\infty}(\R^n)= \bigcap_{\mu, m \in \R} L^{\mu, m}(\R^n)$,  the set of the so-called smoothing operators, coincides with the set of operators with kernel in $\SX(\R^{2n})$: these constitute the residual elements of the $SG$-calculus. 
A main tool, in particular, is the subclass of $SG$-classical operators, whose symbols, belonging to $SG_\cl^{\mu, m}(\R^n)
\subset SG^{\mu, m}(\R^n)$, admit a double asymptotic expansion in terms which are homogeneous w.r.t. the 
$\xi$-variable or the $x$-variable, respectively. More precisely, for $a \in SG_\cl^{\mu, m}(\R^n)$ we denote 
\renewcommand{\theenumi}{\roman{enumi}}
\begin{enumerate}
	\item by $a_{\mu-j, \cdot}$ the terms of order $(\mu-j,m)$, homogeneous w.r.t. the $\xi$-variable,
	such that, for a fixed $0$-excision function $\chi=\chi(\xi)$, $a\sim \sum_j \chi \, a_{\mu-j,\cdot}\!\!\!\mod SG^{-\infty,m}(\R^n)$;
	\item by $a_{\cdot, m-k}$ the terms of order $(\mu,m-k)$, homogeneous w.r.t. the $x$-variable,
	such that, for a fixed $0$-excision function $\omega\!=\!\omega(x)$,
	$a\sim \sum_k \omega \,a_{\cdot,m-k}\!\!\!\mod SG^{\mu,-\infty}(\R^n)$;
	\item by $a_{\mu-j, m-k}$ the terms of order $(\mu-j,m-k)$, homogeneous w.r.t. the $\xi$-variable
	and the $x$-variable, such that, for the same $0$-excision functions,
	$\chi\,a_{\mu-j,\cdot}\sim\sum_k \chi\,\omega\, a_{\mu-j,m-k}\mod SG^{\mu-j,-\infty}$
	and
	$\omega\,a_{\cdot,m-k,}\sim\sum_j \chi\, \omega\,a_{\mu-j,m-k}\mod SG^{-\infty,m-k}$.
\end{enumerate}
The set $\{\op(a)\mid a \in SG^{\mu, m}_\cl(\R^n)\}$ is denoted by $L^{\mu, m}_\cl(\R^n)$:
typical examples are differential operators with polynomial coefficients.
For details about the $SG$-calculus and its properties, see, e.g., \cite{BTh}, \cite{CO}, \cite{ES97}, \cite{MSS06}, \cite{PA72}, \cite{SC87},
and the references quoted therein.

Here we prove a \textit{regularised version} of \eqref{KKWTHM} on $\R^n$, $n\geq 4$, endowed with a suitable Riemannian metric $g$ and corresponding 
induced measure $dx$, 
see Sections \ref{sec:2} and \ref{sec:3} below. We consider a generalised positive Laplacian $\Delta= \nabla^* \nabla + \mathscr{K}$, defined on a Hermitian vector bundle $E$ with connection $\nabla$, where $\mathscr{K}$ is a symmetric endomorphism field. Moreover,  
in order to define and make use of its fractional powers, we assume that the spectrum $\sigma(\Delta)$ of $\Delta$ lies outside a sector of the complex plane, with at most the exception of the origin. Finally, we obtain
\begin{equation}
	\label{eq:wresgen}
	\widehat{\wres}(\Delta^{-\frac{n}{2}+1})= \frac{(n-2)}{\Gamma(\frac{n}{2}) (4 \pi)^\frac{n}{2}} \fint 
	\left[\frac{\textnormal{Rk}(E)}{6}s(x)- \textnormal{Trace}(\mathscr{K}_x)\right]dx,
\end{equation}
where $\widehat{\wres}(\cdot)$ is a generalised Wodzicki residue, defined for $SG$-operators of order $0$ w.r.t. the $x$-variable, and $\fint$ denotes the \textit{finite part integral}\footnote{A similar result was obtained in \cite{BTh} by direct evaluation.}. 

The paper is organised as follows. Sections \ref{sec:1} and \ref{sec:2} are devoted to illustrate the definitions of the finite part integral of classical symbols, and of regularised trace and $\zeta$-function for $SG$-operators, respectively.
In Section \ref{sec:3} we give the proof of \eqref{eq:wresgen} under the assumption that, if the origin belongs to $\sigma(\Delta)$, it is an isolated point of $\sigma(\Delta)$, and, finally, we obtain a regularised version of \eqref{KKWTHM}.

\section*{Acknowledgements}
The authors would like to thank R. Mazzeo, L. Rodino, and E. Schrohe, for useful suggestions and discussions.

\section{Finite part integral}
\label{sec:1}
The \emph{finite-part integral}, introduced in \cite{MSS06b}, gives a meaning to the integral of a classical symbol $a$, and coincides with the usual integral when $a\in L^1(\R^n)$. $dS$ denotes the usual measure on $|x|=1$, induced by the Euclidean metric on $\R^n$, while, in this Section, $dx$ denotes the standard Lebesgue measure on $\R^n$.

\begin{Def}
Let $a$ be an element of the classical H\"ormander symbol class $S^{m}_\cl(\R^n)$, that is, 
\begin{enumerate}
	\item $a\in C^\infty(\R^n)$ and $\forall x\in\R^n \; |D^\alpha a(x)|\le C_\alpha (1+|x|)^{m-|\alpha|}$;
	\item $a$ admits an asymptotic expansion in homogeneous terms $a_{m-j}$ of order $m-j$: 
	explicitly, for a fixed $0$-excision function
	$\omega$ and all $N\in\N$,
	\[
		a- \sum_{j=0}^{N-1}\omega \, a_{m-j} \in S^{m-N}(\R^n).
	\]
\end{enumerate}
Then:
\begin{itemize}
\item[-] if $m \in \Z$, set
\begin{align*}
\fint  a(x)\,dx&:= \lim_{\rho \to \infty}\left[ \int_{|x|\leq \rho}a(x) \, dx - \sum_{j=0}^{m-n} \int_{|x|\leq \rho} a_{m-j}(x) \, dx\right]
\\
&= \lim_{\rho \to \infty} \left[ \int_{|x|\leq \rho}a(x) \, dx- \sum_{j=0}^{m-n} \frac{\beta_j}{n+m-j}\rho^{n+m-j}-\beta_{n+m} \log \rho\right]
\end{align*}
where
\begin{equation}
\label{beta}
\beta_j:= \int_{|x|=1} a_{m-j} \, dS;
\end{equation}
\item[-] if $m \not \in \Z$, set
\begin{equation}
\fint  a(x)\,dx:= \lim_{\rho \to \infty}\left[ \int_{|x|\leq \rho}a(x)\, dx - \sum_{j=0}^{[m]-n-1} \int_{|x|\leq \rho} a_{m-j}(x)\, dx\right].
\end{equation}
\end{itemize}
\end{Def}
\noindent From the above Definition it is clear that if $a \in L^1(\R^n)$ the finite part integral is equivalent to the standard integral. If $m \notin \Z$ the finite part integral coincides with the Kontsevich-Vishik density \cite{KV94}, \cite{KV95}.
\begin{rem}
If one considers the radial compactification of $\R^n$ to $\mathbb{S}^n_+$, namely
\[
rc\colon \R^n \to S^n_+\colon x=(x_1, \ldots, x_n)\mapsto y=\left[\frac{x_1}{(1+ |x|^2)^{\frac{1}{2}}}, \ldots, \frac{x_n}{(1+ |x|^2)^{\frac{1}{2}}}, \frac{1}{(1+ |x|^2)^{\frac{1}{2}}}\right],
\]
and chooses $y_{n+1}$ as boundary defining function on $\mathbb{S}^n_+$, its composition with $rc$ coincides in the interior with $\frac{1}{(1+ |x|^2)^{\frac{1}{2}}}$, $x=rc^{-1}(y)\in\R^n$. Then
\[
\fint a(x)\,dx=\, {^R\hspace{-2mm} \int_{\mathbb{S}^n_+}a (rc^{-1}(y))\,dS(y)}
\]
where the right hand side 
is defined as the term of order $\epsilon^0$ in the asymptotic expansion of
\[
\int_{\mathbb{S}^n_+\cap\{y_{n+1}\geq \epsilon\}} a(rc^{-1}(y))\,dS(y),  \quad \epsilon \searrow 0.
\]
$\displaystyle{^R\hspace{-2mm}  \int_{\mathbb{S}^n_+}f\,dS}$
is called \emph{Renormalised integral}, see \cite{AL07} and the references quoted therein
for its precise definition, properties and applications.
\end{rem}

\section{Regularised trace and regularised $\zeta$-function}
\label{sec:2}
In the sequel we will often make reference to sectors of the complex plane with vertex at the origin, that is, subsets of
$\C$ given by $\Lambda=\{z\in\C\mid -\pi+\theta\le\arg(z)\le\pi-\theta\}$, $0<\theta<\pi$, as in the next picture.
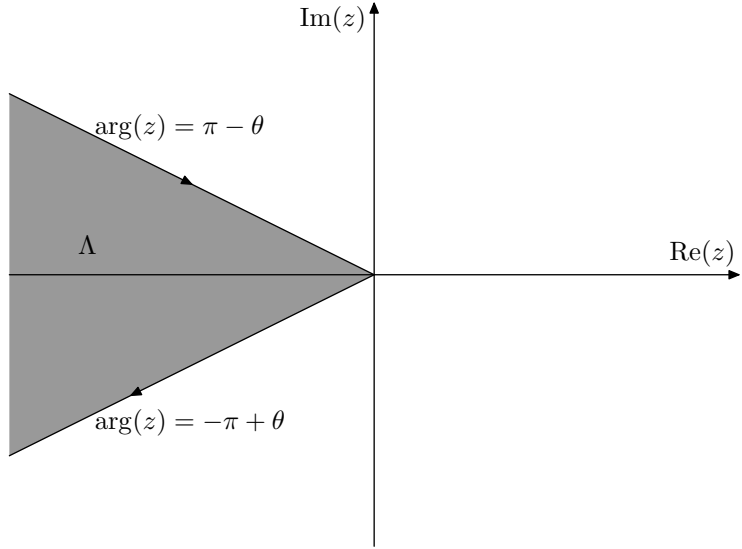
\begin{center}
	\includegraphics{settore.1}
\end{center}

\noindent
The definition of $\Lambda$-elliptic operator is the standard one, here given for operators defined through matrix-valued symbols,
whose spectrum we denote by $\sigma(a(x,\xi))$:
\begin{Def}
\label{elldeb}
The operator $A \in L^{\mu,0}(\R^n)$ is  $\Lambda$-elliptic if there exists a constant $R > 0$ such that
\begin{equation}
\label{1ell}
\sigma(a(x,\xi)) \cap \Lambda =\emptyset\quad \forall |\xi|\geq R, \;\; \forall x \in \R^n
\end{equation}
and
\begin{equation}
\label{2ell}
(a(x,\xi)-\lambda)^{-1} \in SG^{-\mu, 0}(\R^n) \quad \forall  |\xi|\geq R,\;\; \forall x \in \R^n,\;\;\forall\lambda \in \Lambda.
\end{equation}
\end{Def}

\noindent It is well known that, if an operator $A$ is $\Lambda$-elliptic, we can build a \textit{weak} parametrix $B(\lambda)$ such that
\begin{equation}
\label{weak}
\begin{split}
&B(\lambda) \circ (A-\lambda I)= \id + R_1(\lambda),\\
&(A-\lambda I) \circ B(\lambda)=\id + R_2(\lambda), \quad R_1, R_2 \in L^{-\infty,0}(\R^n).
\end{split}
\end{equation}
Moreover 
\begin{equation}
\label{secto}
\begin{split}
&\lambda \,B(\lambda) \in L^{-\mu,0}(\R^n),\\
&\lambda^2 \,\big[(A-\lambda I)^{-1}- B(\lambda)\big] \in L^{-\infty, 0}(\R^n), \quad \forall \lambda \in \Lambda\setminus\{0\}.
\end{split}
\end{equation}

From now on, $\mu>0$ and $A$ is considered as an unbounded operator with dense domain
$D(A)=H^\mu(\R^n)\hookrightarrow L^2(\R^n) \to L^2(\R^n)$. 
To define the complex powers of a $\Lambda$-elliptic operator $A$, we assume that the following property holds for
its spectrum $\sigma(A)$:
\newcounter{tmpeq}
\setcounter{tmpeq}{\value{equation}}
\renewcommand{\theequation}{A\arabic{equation}}
\setcounter{equation}{0}
\begin{equation}
\label{sass}
\sigma(A)\cap \{\Lambda \setminus \{0\}\}=\emptyset \text{ and the origin is at most an isolated point of $\sigma(A)$}. 
\end{equation}
\renewcommand{\theequation}{\thesection.\arabic{equation}}
\setcounter{equation}{\value{tmpeq}}

\begin{Prop}
Let $A\in L^{\mu, 0}(\R^n)$, $\mu>0$, be a $\Lambda$-elliptic operator that satisfies Assumption \eqref{sass}. The complex power $A^z$, $\Re(z)<0$,
can be defined as
\begin{equation}
\label{exp}
A^z:= \frac{1}{2\pi i}\int_{\partial^+ \Lambda_\epsilon} \lambda^z (A-\lambda I)^{-1}d\lambda.
\end{equation}
where $\Lambda_\epsilon=\Lambda \cup \{z \in \C \mid |z|\leq \epsilon\}$, with $\epsilon>0$ chosen such that $\sigma(A) \cap \{\Lambda_\epsilon \setminus \{0\}\} = \emptyset$ and $\partial^+\Lambda_\epsilon$ is the (positively oriented) boundary of
$\Lambda_\epsilon$. 

\end{Prop}
\begin{proof}
By the definition of $\Lambda$-elliptic operator, we know that $(A-\lambda I)^{-1}$ exists for all
$\lambda \in \Lambda\setminus \{0\}$. Moreover, by \eqref{secto} we have that $A$ is sectorial, so the integral \eqref{exp}
converges in $\mathcal{L}(L^2(\R^n)$). 
\end{proof}

\begin{rem}
The definition of $A^z$ is then extended to arbitrary $z\in\C$ in the standard way, that is $A^z:=A^{z-j}\circ A^j$,
where $j\in\Z_+$ is chosen so that $\Re(z)-j<0$, see, e.g., \cite{BC09}, \cite{MSS06}, \cite{SE67}.
\end{rem}

\begin{Thm}
\label{thm:az}
Let $A\in L^{\mu, 0}(\R^n)$, $\mu>0$, be $\Lambda$-elliptic and satisfy Assumption \eqref{sass}. Then, $A^z \in L^{\mu z, 0}(\R^n)$.
Moreover, if $A$ is $SG$-classical then $A^z$ is still $SG$-classical\footnote{To define the class of symbols $SG^{z, w}(\R^n)$,
$z,w\in\C$, we just have to substitute $\Re(z)$ and $\Re(w)$ in place of $m,\mu$, respectively, in the estimates \eqref{controllo}.}.
\end{Thm}
\begin{rem}
In order to define the symbol of $A^z$, the resolvent $(A-\lambda I)^{-1}$ can be approximated with the weak parametrix
$B(\lambda)$ defined in \eqref{weak}. In this way, a symbol for $A^z$ can be computed, modulo smoothing operators w.r.t. the $\xi$-variable. $A^z$ can then be considered as an element of the algebra $\mathscr{A}$ given by
\begin{equation}
\label{algregxi}
\mathscr{A}:=\bigcup_{\mu \in \Z} L^{\mu,0}(\R^n)/ L^{-\infty,0}(\R^n).
\end{equation}
\end{rem}
\noindent
The proof of the Theorem \ref{thm:az} has been given in \cite{MSS06}.

From here on, $dx$ will denote the measure induced on $\R^n$ by a smooth Riemannian metric
$g=(g_{jk})$. In order to obtain a result similar to \eqref{KKWTHM} we have to impose some condition on $g$, namely\footnote{In the $\flat$-calculus setting, this condition implies that the underlying metric is polyhomogeneous: this is used, for instance, in \cite{AlMa10}.}
\setcounter{tmpeq}{\value{equation}}
\renewcommand{\theequation}{A\arabic{equation}}
\setcounter{equation}{1}
\begin{equation}
\label{assu}
\begin{array}{rl}
 & \text{$g$ is a matrix-valued $SG$-classical symbol of order $(0,0)$. }
\end{array}
\end{equation}
\renewcommand{\theequation}{\thesection.\arabic{equation}}
\setcounter{equation}{\value{tmpeq}}
If $A\in L^{\mu, m}(\R^n)$ is trace class, that is $\mu<-n, m<-n$, we can define its trace
\[
TR(A):=\int K_A(x,x)\,dx,
\]
where $K_A(x,x)$ is the kernel of $A$ restricted to the diagonal. The concept of regularised trace, valid for classical
$SG$-operators under less restrictive hypotheses on the order, has been introduced in \cite{MSS06b}, using the finite
part integral defined in the previous Section:
\begin{Def}
Let $A \in L^{\mu, m}_\cl(\R^n)$ be such that $\mu <-n$. We define the regularised trace of $A$ as
\begin{equation}
\label{traccia}
\widehat{TR}(A):= \fint K_{A}(x,x)\,dx.
\end{equation}
\end{Def}
\begin{rem}
Note that the condition $\mu<-n$ implies that $K_A(x,x)$ is indeed a function and that the finite part integral \eqref{traccia} is well defined.  
\end{rem}

\noindent
Now, using the regularised integral, we can give the definition of regularised $\zeta$-function:

\begin{Def}
Let $A \in L_{\cl}^{\mu, 0}(\R^n)$, $\mu>0$, be a $\Lambda$-elliptic operator that satisfies \eqref{sass}; then we define
\begin{equation}
\label{fzeta}
\hat{\zeta}(A, z):= \widehat{TR}(A^{-z})= \fint K_{A^{-z}}(x,x)\,dx, \quad \Re(z)>\frac{n}{\mu},
\end{equation}
where $K_{A^{-z}}(x,x)$ is the kernel of the operator $A^z$.
\end{Def}
\noindent It is simple to prove that $\hat{\zeta}(A,z)$ is holomorphic for $\Re(z)>\frac{n}{\mu}$, in view of the fact that
the hypotheses imply that the kernel $K_{A^z}(x,x)$ is a function.
As in the case treated in \cite{SE67}, we can look for meromorphic extensions of $\hat{\zeta}(A,z)$.
\begin{Thm} 
Let $A \in L_\cl^{\mu, 0}(\R^n)$, $\mu>0$, be a $SG$-operator that admits complex powers. Then the function $\hat{\zeta}(A,z)$ can be extended as a meromorphic function with, at most, poles at the points $z_j=\frac{n-j}{\mu}$, $j \in \N$.
\end{Thm}

Following the idea of M. Wodzicki \cite{WO84}, \cite{KA89}, we can now introduce a regularised version of non-commutative residue.
\begin{Def}
Let $A \in L_{\cl}^{\mu, 0}(\R^n)$, $\mu>0$,  be a $\Lambda$-elliptic operator that satisfies \eqref{sass}. We define the regularised Wodzicki residue of $A$ as 
\[
\widehat{\wres}(A):= \mu \, \res_{z=-1} \hat{\zeta}(A,z).
\]
\end{Def}
\noindent In the case $\mu \in \N$, using the explicit expression of the regularised integral and of the residues of $\hat{\zeta}(A,z)$, we get
\begin{equation}
\label{regres}
\widehat{\wres}(A)\!=\! \frac {1}{(2\pi)^{n}}\!\lim_{\rho \to \infty} \left[ \int_{|x|\leq \rho}\!\int_{|\xi|=1} \!\!a_{-n, \cdot}(x,\xi) \, dS(\xi)dx\!- 
\!\!\!\sum_{j=0}^{\mu+n-1} \!\!\frac{\beta_j}{n-j}\rho^{n-j}\! -\!\beta_{\mu+n} \log \rho\right]
\end{equation}
where 
\[
\beta_{j}= \int_{|x|=1} \int_{|\xi|=1} a_{n-j,\cdot}\, dS(\xi) \widetilde{dS}(x),
\] 
$\widetilde{dS}(x)$ the metric induced by $g$ on $|x|=1$.
The case $\mu \not \in \Z$ is not very interesting, since then $\widehat{\wres}(A)$ always vanishes, due to the fact that, in this case, the kernel $K_{A^{-z}}(x,x)$ has no poles at $z=-1$.
The residue $\widehat{\wres}(\cdot)$ also vanishes on smoothing operators w.r.t. the $\xi$-variable, so it is well defined on the
algebra $\mathscr{A}$. Incidentally,
let us notice that the expression \eqref{regres} is analogous to the functional
$\res_\psi(A)$ defined by F. Nicola in \cite{NI03}, by means of holomorphic operator families.

\section{A Kastler-Kalau-Walze type Theorem on $\R^n$}
\label{sec:3}
First, we restrict to the case of $\R^4$ and consider the classical Atiyah-Singer Dirac operator $\Dir$ acting on the spinor bundle $\Sigma \R^4$.
 If the metric on $\R^4$ satisfies Assumption \eqref{assu}, it is immediate to verify that $\Dir \in L^{1,0}_\cl$. Let $\Dir^{-2}$ denote a \emph{weak} parametrix of the square of the Dirac operator, that is $\Dir^2 \circ \Dir^{-2}= I + R$, $R \in L^{-\infty, 0}$. The calculus implies that $\Dir^{-2} \in L_\cl^{-2, 0}$.
Via direct computation, following the idea of D. Kastler \cite{KA95}, it is possible to compute $a_{-4,.}(x, \xi)$, the term of of order $-4$ in the asymptotic expansion w.r.t. the $\xi$-variable of the symbol of $\Dir^{-2}$.
Evaluating the integral on the sphere w.r.t. the $\xi$ variable one gets
\[
\int_{|\xi|=1}a_{-4,.}(x, \xi)\,dS(\xi) = -\frac{1}{24 \pi^2} s(x).
\] 
So we have that 
\begin{equation}
\label{KKWtesi}
\widehat{\wres}(\Dir^{-2})=- \frac{1}{24 \pi^2}\fint s(x)\,dx.
\end{equation}
The proof of \eqref{KKWtesi} is contained in \cite{BTh}. Let us notice the slight abuse of notation in \eqref{KKWtesi},
due to the fact that, in general, $\Dir^{-2}$ does not satisfy Assumption \eqref{sass}: anyway, we can use \eqref{regres}
as a definition of $\widehat{\wres}(\Dir^{-2})$ in this case.

In order to obtain a generalisation of \eqref{KKWtesi} to higher dimensions and to more general operators, the direct approach seems to be rather cumbersome. For this reason, we follow an idea of T. Ackermann \cite{AC96} and exploit the properties of the asymptotic expansion of the heat kernel of generalised Laplacians.

As explained in the previous Section, if $A \in L_\cl^{\mu, 0}(\R^n)$, $\mu>0$, is $\Lambda$-elliptic and satisfies Assumption \eqref{sass}, we can define the complex powers of $A$ and the heat semigroup $e^{-tA}$ as well:
\[
e^{-tA}:= \frac{i}{2\pi }\int_{\partial^+ \Lambda_\epsilon} e^{-t\lambda}(A-\lambda I)^{-1} d\lambda.
\]
In \cite{MSS06b} it has been proved that $e^{-tA}$ is a $SG$-operator belonging to $L^{-\infty, 0}(\R^n)$, so we can also consider the regularised heat trace $\widehat{TR}(e^{-tA})$. There is a deep link between regularised heat trace and $\hat{\zeta}$-function:
\begin{Thm}
Let $A \in L_\cl^{\mu, 0}(\R^n)$, $\mu>0$, be an operator that admits complex powers. Then, for suitable constants
$c_{kl}=c_{kl}(A)$, the following two asymptotic expansions hold:

\begin{align}
\label{expzeta}
\Gamma(z)\,\hat{\zeta}(A,z) &\sim \sum_{k=0}^{\infty}\sum_{l=0}^1c_{kl}\left(z- \frac{n-k}{\mu}\right)^{-l+1}
\\
\label{heatas}
\widehat{TR}(e^{-tA}) &\sim \sum_{k=0}^{\infty}\sum_{l=0}^1 (-1)^l c_{kl}t^{-\frac{n-k}{\mu}}log^l t, \quad t\searrow 0.
\end{align}

\end{Thm}
\begin{proof}
The statement follows by adapting the arguments given in \cite{MSS06b}
to the present situation. A main role in the proof is played by an abstract theorem by G. Grubb and R. Seeley \cite{GS96}, connecting $\zeta$-functions and heat traces.
\end{proof}

Let us now consider a generalised positive Laplacian $\Delta\in L^{-2, 0}_\cl(E)$, where $E$ is a Hermitian vector bundle on
$\R^n$ with connection $\nabla$, that is
\[
\Delta= \nabla^* \nabla+ \mathscr{K}, \quad \mathscr{K}\in C^{\infty}(\textnormal{End}(E)) \mbox { symmetric endomorphism field}.
\]
We require that $\Delta$ satisfies Assumption \eqref{sass}: in this way, we can define $e^{-t\Delta}$ as above. In the case of closed manifolds, it is well known (see, e.g., \cite{AB02}) that
\begin{equation}
\label{kernelas}
K_{e^{-t \Delta}}(x,x)=k_t(x,x) \sim (4\pi t)^{-\frac{n}{2}}\left[ 1\cdot \textnormal{Id}_{E} + \big(\frac{1}{6}s(x) \textnormal{Id}_E - \mathscr{K}_x\big)t+ O(t^2)\right], \quad t\searrow 0.
\end{equation}
where $s(x)$ is the scalar curvature of the underlying manifold and the remainder term depends only on the connection and on the endomorphism field. The asymptotic expansion \eqref{kernelas} also holds in the case of manifolds with cylindrical ends, since the computations are completely analogous and purely local, see \cite{AL07}. The evaluation of the first term of the asymptotic expansion can be found in \cite{MP02}: the expression of the second term then follows, using the properties of generalised Laplacians.
In view of our hypotheses, the right hand side of \eqref{kernelas} is a classical $SG$-symbol: then, we obtain
\begin{equation}
\begin{split}
\label{trheat}
\widehat{TR}&(e^{-t\Delta}) \sim \\
&\hspace{1mm}
(4 \pi t)^{-\frac{n}{2}}\left\{ \fint \textnormal{Rk}(E)\,dx + t\!\fint\left[\frac{\textnormal{Rk}(E)}{6}s(x) -\textnormal{Trace}( \mathscr{K}_x)\right]
dx+ O(t^2) \right\}, \; t\searrow 0.
\end{split}
\end{equation}
Since, trivially, when $h$ is a meromorphic function with a simple pole in $z_0$, the function $\tilde{h}(z)=h(c z)$, $c \in \R$, is a meromorphic function with a simple pole in $\frac{z_0}{c}$ and
\[
\res_{w=\frac{z_0}{c}} \tilde{h}= \frac{1}{c}\res_{z=z_0}h,
\] 
we also have that
\begin{equation}
\label{rag}
\begin{split} 
\widehat{\wres}(\Delta^{-\frac{n}{2}+1})&=(2-n)\, \res_{z=-1} \hat{\zeta}(\Delta^{-\frac{n}{2}+1}, z)=2 \, \res_{z=\frac{n-2}{2}} \hat{\zeta}(\Delta,z)
\\
&= 2 \,\Gamma \left(\frac{n-2}{2}\right)^{-1}c_{2,0}(\Delta),
\end{split}
\end{equation}
where  $c_{2,0}(\Delta)$ is coefficient of the term of order $t^{-\frac{n-2}{2}}$ in the asymptotic expansion \eqref{heatas}.
Finally, by \eqref{trheat} and the properties of $\Gamma(z)$,
\begin{equation}
\label{KKW}
\widehat{\wres}(\Delta^{-\frac{n}{2}+1})= \frac{n-2 }{\Gamma(\frac{n}{2}) (4 \pi)^{\frac{n}{2}}} \fint \left[ \frac{\textnormal{Rk}(E)}{6}s(x)- \textnormal{Trace}(\mathscr{K}_x) \right]dx.
\end{equation}
\begin{rem}
Assumption \eqref{sass} does not imply that $\Delta$ is invertible, since we allow the origin to be an isolated point of
$\sigma(\Delta)$. In view of this, the operator $\Delta^{-\frac{n}{2}+1}$ has to be interpreted in the sense of the complex powers defined above.
\end{rem}

If we consider  a generalised Laplacian $\Delta$, then its principal homogeneous symbol is $g^{jk}(x) \xi_j\xi_k=|\xi|^2>0$, $\xi\not=0$.
$\Delta$ turns out to be
always $\Lambda$-elliptic with respect to a suitable sector of the complex plane, while $\sigma(\Delta)$ can admit the origin as an accumulation point. For example, it is well known that the classical Atiyah-Singer Dirac operator on $\R^n$, endowed with the canonical Euclidean metric, has no point spectrum, but the essential spectrum is the whole real line. In this case Assumption \eqref{sass} of course fails to be true\footnote{For further properties of the Dirac spectrum on open manifolds, the reader can refer, for instance, to the monograph by N. Ginoux \cite{GI09}.}.
A simple example such that Assumption \eqref{sass} is satisfied can be built in the following way. Let us consider a general Dirac operator $D$, defined on a Clifford bundle $E$ over $\R^n$: $D^2$ is then a generalised Laplacian and a non-negative operator. If we consider $D^2_\epsilon= D^2+\epsilon I$, we obtain an invertible generalised Laplacian, that clearly satisfies  \eqref{sass}. If we consider the classical Atyiah-Singer Dirac operator $\Dir$, formula \eqref{KKW} turns to
\begin{equation}
\widehat{\wres}((\Dir^2_\epsilon)^{-\frac{n}{2}+1})=\frac{(n-2) 2^{[ \frac{n}{2}]}}{\Gamma(\frac{n}{2}) (4 \pi)^\frac{n}{2}}  \left( - \frac{1}{12} \fint s(x) \,dx
-\epsilon\fint  dx\right).
\end{equation}

On the other hand, a natural example of a metric on $\R^n$ which can satisfy Assumption \eqref{assu} is an asymptotically flat one. In General Relativity, such an hypothesis on the metric is commonly assumed (e.g., in order to define the ADM-mass).
Explicitly, we can consider a metric $g$ such that, for a constant $\alpha>0$,
\[
g_{jk}(x)- \delta_{jk} = O(|x|^{-\alpha}) \mbox{ outside a compact set } K\subset\R^n.
\]
Moreover, restricting ourself to $\R^4$, if $\alpha >2 $ the scalar curvature $s(x)$ is integrable: in this case, \eqref{KKWtesi} becomes
\[
\widehat{\wres}(\Dir^{-2})= -\frac{1}{24 \pi^2} \int s(x)\,dx.
\] 

The method above can be used to treat also the case of manifolds with cylindrical ends, using the contents of \cite{BC09}: one defines in this setting a regularised Wodzicki Residue and exploits its connection with the zeta function. The asymptotic expansion of the heat kernel as $t \searrow 0$ is locally defined, so, using suitable regularised integrals, see \cite{BC09}, the results can be generalised to those manifolds in this class which admit a spin structure. To keep this exposition at a reasonable length, we omit here the details. 

\bibliographystyle{abbrv}

\def\cftil#1{\ifmmode\setbox7\hbox{$\accent"5E#1$}\else
  \setbox7\hbox{\accent"5E#1}\penalty 10000\relax\fi\raise 1\ht7
  \hbox{\lower1.15ex\hbox to 1\wd7{\hss\accent"7E\hss}}\penalty 10000
  \hskip-1\wd7\penalty 10000\box7}

\end{document}